\documentclass[11pt,a4paper]{amsart}

\usepackage{epsfig} 
\usepackage{times} 
\usepackage{amsmath} 
\usepackage{amssymb}  
\usepackage{xcolor}
\usepackage[ruled]{algorithm2e}
\usepackage[noend]{algpseudocode}
\usepackage{graphicx}
\usepackage{float}
\usepackage{afterpage}
\usepackage{multicol}
\usepackage[english]{babel}
\usepackage{tikz}

\newtheorem{theorem}{Theorem}

\newtheorem{definition}[theorem]{Definition}
\newtheorem{remark}[theorem]{Remark}

\newtheorem{assumption}[theorem]{Assumption}

\usepackage[centering]{geometry}
\usepackage{placeins}

\begin{document}
	
%
%
%
%
\title[Conditions for strict dissipativity]{Conditions for strict dissipativity of
	infinite-dimensional generalized linear-quadratic problems}
\author[L. Gr\"une, D.\ Muff, M.\ Schaller ]{ Lars Gr\"une$^1$, David Muff$^1$ and Manuel Schaller $^2$}
	\thanks{}
		\thanks{$^{1}$University of Bayreuth, Department of Mathematics, Germany
			(e-mail: \{lars.gruene,david.s.muff\}@uni-bayreuth.de).} %
	\thanks{$^{2}$Technische Universit\"at Ilmemau, Institute for Mathematics, Germany (e-mail: manuel.schaller@tu-ilmenau.de). M. Schaller was funded by the German Research Foundation (DFG; project numbers 289034702 and 430154635).}%
\maketitle
\begin{abstract}                
	We derive sufficient conditions for strict dissipativity for optimal control of linear evolution equations on Hilbert spaces with a cost functional including linear and quadratic terms. We show that strict dissipativity with a particular storage function is equivalent to ellipticity of a Lyapunov-like operator. Further we prove under a spectral decomposition assumption of the underlying generator and an orthogonality condition of the resulting subspaces that this ellipticity property holds under a detectability assumption. We illustrate our result by means of an example involving a heat equation on a one-dimensional domain.
	
	\smallskip
	\noindent \textbf{Keywords.} 	optimal control, dissipativity, infinite-dimensional system, detectability, ellipticity, Lyapunov inequality, necessary optimality condition
\end{abstract}


\section{Introduction}

Since their introduction by Jan Willems in \cite{Will72a,Will72b}, dissipativity and strict dissipativity have been recognized as important systems theoretic properties. These properties formalize a kind of energy balance property, expressed via the stored energy in form of a storage function depending on the state of the system, and the supplied energy in form of a supply function depending on the input and possibly in the output of the system. In optimal control-related applications of dissipativity, the supply function is often linked to the running cost or stage cost of the optimal control problem. While dissipativity formalizes the fact that a system cannot store more energy than supplied from the outside, strict dissipativity in addition requires that at each time a certain amount of the stored energy is dissipated to the environment. 

One of the motivations to study dissipativity like properties is that they are directly linked to stability considerations. Particular forms of dissipativity like, e.g., passivity naturally serve as tools for the design of stabilizing controllers, see \cite{ByIW91,Scha00}. In recent years, it was discovered that dissipativity properties of optimal control problems form an important ingredient for understanding the stability behavior of economic model predictive control (MPC) schemes, see, e.g.,  \cite{DiAR10,AnAR12,Grue13,GruS14}. In essence, they allow for the construction of a Lyapunov function from an optimal value function also in case the cost of the optimal control problem under consideration is not positive definite. Moreover, they are closely related to the existence of steady states at which the system is optimally operated, see \cite{Muel14,MuAA15} and \cite{MuGA15}, and to the so-called turnpike property at an optimal equilibrium, see \cite{GruM16}.

Important aspects of the relation between dissipativity and optimal control were studied already in the early days of dissipativity theory in \cite{Will71}. Yet, the investigation of when a system is strictly dissipative with a supply rate derived from the cost function of an optimal control problem is a relatively recent research topic, motivated by the MPC applications mentioned above. For generalized linear quadratic problems, i.e., problems with linear dynamics and cost functions consisting of linear and quadratic terms, this question was comprehensively studied in \cite{GruG18,GruG21} for finite dimensional systems in discrete and continuous time. Particularly, the results in these papers establish a close relation between strict dissipativity and detectability-like properties, when the output is chosen corresponding to the quadratic part of the cost function. Interestingly, detectability (among other properties) also plays an important role for establishing exponential turnpike properties for a large class of infinite-dimensional linear-quadratic optimal control problems, see \cite{GrSS20}. However, what still remains to be clarified in the infinite-dimensional setting is the role of strict dissipativity.

This paper gives some first insights into this subject. We show that, under suitable assumptions on the system, strict dissipativity is equivalent to the ellipticity of a certain operator resulting from a particular Lyapunov equation. Based on this result, we further show that exponential detectability implies this ellipticity condition and hence strict dissipativity. The results are doubtlessly preliminary in the sense that there is hope that the assumptions we impose could be relaxed and that --- as in the finite-dimensional case --- it should also be possible to establish necessity of detectability for strict dissipativity. Yet, we believe that our results provide a useful first step in clarifying the role of strict dissipativity for infinite-dimensional linear-quadratic problems. 

The rest of the paper is organized as follows. Section 2 defines the setting, in Section 3 the ellipticity characterization is proved and in Section 4 the relation to detectability is shown. Section 5 illustrates our results by an example and Section 6 concludes the paper.

\section{Problem setting}

Let $(X,\langle \cdot,\cdot\rangle)$ and $(U,\langle \cdot,\cdot\rangle_U)$ be Hilbert spaces with norms $\|\cdot\|$ and $\|\cdot\|_U$ respectively. Further let $A:D(A)\subset X \to X$ be the closed and densely defined generator of a strongly continuous semigroup $(\mathcal{T}(t))_{t\geq 0}$, $B\in L(U,X)$ a control operator, $T>0$ be a time horizon and $x_0\in X$ an initial datum.

We denote by $\rho(A) := \{s \in \mathbb{C}\,|\,sI-A\,:\,D(A)\to X \text{ is bijective}\}$ the resolvent set of $A$ and for any $s\in \rho(A)$ by $(sI-A)^{-1} \in L(X,X)$ the resolvent operator, cf.\ \cite[Definition A.4.1]{Curtain1995}.

The stage cost $\ell:X\times U\to \mathbb{R}$ is defined via 
\begin{align*}
\ell(x,u) = \langle x,Qx\rangle  + \langle u,Ru\rangle_{U}  + \langle s,x\rangle + \langle v,u\rangle_U,
\end{align*}
where $Q=C^*C$ for some $C\in L(X,Y)$ and $(Y,\langle\cdot,\cdot\rangle_Y)$ is another Hilbert space with norm $\|\cdot\|_Y$. Further, $R\in L(U,U)$ is elliptic, i.e. there is $c_R > 0$ with $\langle u(t),Ru(t)\rangle_{U} \geq c_R \|u\|^2_U$ for all $u\in U$. For the linear terms we assume that $s\in X$ and $v\in U$. In the following we will tacitly identify the Hilbert spaces $X,U,Y$ with their dual space via the Riesz isomorphism.

We consider the optimal control problem (OCP)
\begin{align}
\label{OCP}
\begin{split}
\min_{u\in L_2(0,T;U)} \int_0^T &\ell(x(t),u(t))\,dt\\
\dot{x}(t) &= Ax(t) + Bu(t) \quad t\in [0,T],\\
x(0)&=x_0.
\end{split}
\end{align}

The mild solution $x\in C(0,T;X)$ of the dynamics in \eqref{OCP} is given by the variation of constants formula
\begin{align*}
x(t) = \mathcal{T}(t)x_0 + \int_0^t \mathcal{T}(t-s)Bu(s)\,ds.
\end{align*}

By $(x_e,u_e)\in D(A)\times U$ we will denote equilibria of the dynamics which satisfy $f(x_e,u_e):=Ax_e+Bu_e=0$.
\begin{definition}[Strict (pre)-dissipativity]
We call the OCP \eqref{OCP} \emph{strictly pre-dissipative} at an equilibrium $(x_e,u_e)\in D(A)\times U$ on a set $\mathcal{X}\subset D(A)$ and $\mathcal{U}\subset U$ if there is a storage function $\lambda\in C^1(X, \mathbb{R})$ bounded on bounded subsets, a function $\alpha \in \mathcal{K}:=\{\alpha:\mathbb{R}^+_0\to\mathbb{R}^+_0\,|\, \alpha \text{ continuous and strictly increasing with } \alpha(0)=0\}$ and an equilibrium $(x_e,u_e)\in D(A)\times U$ such that
\begin{align}
\label{eq:diss_ineq}
D\lambda(x)\left(Ax+Bu\right)\leq \ell(x,u) - \ell(x_e,u_e) - \alpha(\|x-x_e\|)
\end{align}
for all $x\in \mathcal{X}$ and $u\in \mathcal{U}$.

If additionally $\lambda$ is bounded from below, we call \eqref{OCP} \emph{strictly dissipative}.
\end{definition}
We briefly note that \eqref{eq:diss_ineq} can be stated in integral fashion via the fundamental theorem of calculus, i.e.,
\begin{align*}
&\lambda(x_u(T,x_0)) - \lambda(x_0) \\ &\leq \int_0^T \ell(x_u(t,x_0),u(t))-\ell(x_e,u_e) - \alpha(\|x_u(t,x_0)-x_e\|)\,dt,
\end{align*}
where $x_u(\cdot,x_0)\in C(0,T;X)$ is the (mild) solution of the dynamics of \eqref{OCP} emanating from the initial value $x_0\in X$ and control $u\in L_2(0,T;U)$. This formulation also makes sense for mild solutions $x_u(t,x_0)\in C(0,T;X)$ and if the storage function is not differentiable.

An immediate consequence of the above dissipativity inequality $\eqref{eq:diss_ineq}$ is the turnpike property, cf.\ \cite{Damm2014,GruM16,Faulwasser2017}, a very important property of solutions of optimal control problems subject to dynamical systems on long time horizons. Further, dissipativity inequalities of the above kind are central in the analysis of Model Predictive Control schemes, cf.\ \cite{Grue13,GruS14}.

\section{An ellipticity condition for strict (pre)-dissipativity}
\begin{theorem}
\label{thm:1}
Let $P\in L(X,X)$ be self-adjoint and $0\in \rho(A)$. Then there is $q \in X$  such that the problem is strictly pre-dissipative with storage function $\lambda(x) = \langle x,Px\rangle + \langle q, x\rangle$ if and only if there is a constant $m>0$ such that
\begin{align}
\label{eq:main_inequality}
\langle (Q - A^*P - PA)x,x\rangle \geq m\|x\|_X^2 \qquad \forall x\in D(A).
\end{align}
If $P$ is additionally such that $\langle Px,x\rangle \geq c\|x\|^2$ for $c>0$, then the problem is strictly dissipative.
\end{theorem}
\begin{proof}
First, we observe that defining \begin{align*}
\tilde{\ell}(x,u):= -D\lambda(x)(Ax+Bu) + \ell(x,u) - \ell(x_e,u_e),
\end{align*} the dissipativity inequality \eqref{eq:diss_ineq} is equivalent to $\tilde{\ell}(x,u) \geq \alpha(\|x-x_e\|)$. Further, one can straightforwardly compute that
\begin{align*}
\tilde{\ell}(x,u) = \langle x,(Q-A^*P-PA)x\rangle + R(x,u)
\end{align*}
where $R(x,u)$ contains only linear and constant terms in x. Further it is clear that $\tilde{\ell}(x_e,u_e)=0$.

\textbf{\eqref{eq:diss_ineq} $\rightarrow$ \eqref{eq:main_inequality}}: From \eqref{eq:diss_ineq} we deduce that the map $x\mapsto \tilde{\ell}(x,u_e)$ has a strict local minimum in $x_e$. This means that $\frac{\partial}{\partial x} \ell(x_e,u_e)=0$, as well. Further, as $\ell(x,u_e)$ is quadratic in $x$, setting $\delta x = x-x_e$, we have via Taylor's theorem that
\begin{align*}
&\alpha(\|\delta x\|) \stackrel{\eqref{eq:diss_ineq}}{\leq} \ell(x,u_e) \\&= \ell(x_e,u_e) + \frac{\partial}{\partial x} \ell(x_e,u_e)\delta x + \frac12 \frac{\partial^2}{\partial^2 x} \ell(x,u_e)(\delta x,\delta x)\\
& = \langle\delta x,(Q-A^*P-PA)\delta x\rangle.
\end{align*}
As $x$ was chosen arbitrarily, this holds for all $\delta x=x-x_e$ in $D(A)$. For any $\delta x\ne 0$, setting $\delta y = \delta x / \|\delta x\|$, we obtain
\begin{align*}
& \langle \delta x,(Q-A^*P-PA)\delta x\rangle\\
& = \|\delta x\|^2 \langle \delta y,(Q-A^*P-PA)\delta y\rangle\\
& \ge \|\delta x\|^2 \alpha(\|\delta y\|) \, = \, \|\delta x\|^2 \alpha(1).
\end{align*}
This shows the claim with $m=\alpha(1)$.

\textbf{\eqref{eq:main_inequality} $\rightarrow$ \eqref{eq:diss_ineq}}: Let $\gamma\in (0,1]$, set $P_\gamma = \gamma P$ and $Q_\gamma := Q- A^*P_\gamma - P_\gamma A$. Using \eqref{eq:main_inequality} and the positive semidefiniteness of $Q$ we get via straightforward computation that for all $x\in D(A)$
\begin{align}
\label{eq:Q_gamma}
\langle x,Q_\gamma x\rangle  \geq \gamma m \|x\|^2.
\end{align}
We further define for $x\in D(A)$, $u\in U$
\begin{align*}
\ell_\gamma(x,u):=\ell(x,u) - \langle x,P_\gamma\left(Ax+Bu\right)\rangle - \langle \left(Ax+Bu\right),P_\gamma x\rangle
\end{align*}
and compute
\begin{align*}
\ell_\gamma(x,u) &= \langle x,Q_\gamma x\rangle + \langle u,Ru\rangle_U - \langle x,P_\gamma Bu\rangle - \langle P_\gamma Bu,x\rangle \\&+ \langle s,x\rangle  + \langle v,u\rangle_U.
\end{align*}
As $0\in \rho(A)$ we have that $(-A)^{-1}\in L(X,X)$. Consider the OCP
\begin{align*}
\min_{x\in X,u\in U} \ell_\gamma(x,u) \qquad x+A^{-1}Bu = 0
\end{align*}
which we can clearly rewrite as
\begin{align*}
\min_{u\in U} \ell_\gamma(-A^{-1}Bu,u).
\end{align*}
Moreover using coercivity of $R$, we can estimate \begin{align*}
\ell_\gamma(-A^{-1}Bu,u) &\geq c_R \|u\|_U^2 - \gamma m\|A^{-1}B\|^2_{L(U,X)}\|u\|_U^2\\&-2\|A^{-1}B\|_{L(U,X)}\|s\|\|u\|_U-\|v\|_U\|u\|_U.
\end{align*}
Hence, after possibly decreasing $\gamma$ we obtain \begin{align*}
\ell_\gamma(-A^{-1}Bu,u) \geq c\|u\|_U^2
\end{align*} for some $c>0$ and thus, if $\|u\|_U\to \infty$ it follows that $\ell_\gamma(A^{-1}Bu,u)\to \infty$. Hence, the cost functional is radially unbounded and together with reflexivity of $U$, the existence of a minimizer $u_e\in U$ is assured by the classical proof in PDE-Optimization, cf., e.g., \cite[Theorem 1.43]{Hinze09}.
Further, the corresponding optimality conditions yield an adjoint state $p\in X$ such that, defining the modified Lagrange functional
\begin{align*}
L(x,u) = \ell_\gamma(x,u) - \ell_\gamma(-A^{-1}Bu_e,u_e) - \langle p,x+A^{-1}Bu\rangle
\end{align*}
we have by optimality that
\begin{align*}
L(-A^{-1}Bu_e,u_e) \leq L(x,u) 
\end{align*} for all $(x,u)\in X\times U$. 

Moreover, using \eqref{eq:Q_gamma}, \begin{align*}
L''(x,u) \equiv \begin{pmatrix}
Q_\gamma & -P_\gamma B\\
-B^*P_\gamma & R
\end{pmatrix}
\end{align*}
can easily be shown to satisfy $L''(x,u)(\delta z,\delta z) \geq m\|\delta z\|_{X\times U}^2$ for small enough $\gamma > 0$ and  $m>0$, for all $\delta z\in X\times U$.
As, $L(x,u)$ is quadratic, using Taylor series we obtain that
$$L(-A^{-1}Bu_e,u_e) + \tfrac{m}{2}\left\|\begin{smallmatrix}
x-x_e\\u-u_e
\end{smallmatrix}\right\|^2_{X\times U}\leq L(x,u)$$
and hence, defining $\alpha(r)=\tfrac{m}{2} r^2$ and using $L(-A^{-1}Bu_e,u_e)=0$, we conclude
\begin{align*}
L(x,u) > \alpha\left(\left\|\left(\begin{smallmatrix}x - (-A^{-1}Bu_e)\\u-u_e\end{smallmatrix}\right)\right\|_{X\times U}\right)
\end{align*}
for all $(x,u)\in X\times U$. Defining $q:=A^{-\star}p$ and $\lambda(x):=\langle x,P_\gamma x\rangle + \langle q,x\rangle$ we compute for $(x,u)\in D(A)\times U$ that
\begin{align*}
&\ell(x,u) - \ell(x_e,u_e) \\
&= \ell_\gamma(x,u)-\ell_\gamma(x_e,u_e)\\
& \qquad  + \, \langle x,P_\gamma \left(Ax+Bu\right)\rangle + \langle Ax+Bu,P_\gamma x\rangle \\
&=L(x,u) + D\lambda(x)(Ax+Bu)\\
&\geq D\lambda(x)(Ax+Bu) + \alpha(\|x-x_e\|)
\end{align*}
where we used that 
\begin{align*} D\lambda(x)(Ax+Bu) & = \langle x,P_\gamma \left(Ax+Bu\right)\rangle\\
& \quad + \, \langle Ax+Bu,P_\gamma x\rangle  + \langle q,(Ax+Bu)\rangle \end{align*}
and 
$\langle q,(Ax+Bu)\rangle = \langle A^{-\star}p,Ax+Bu\rangle = \langle p,x+A^{-1}Bu\rangle$.
\end{proof}

\section{Sufficient conditions with detectability}
In the previous section, the only restriction on the semigroup resp.\ the generator was that $0\in \rho(A)$. In this part, in order to further characterize dissipativity, or more precisely, the existence of $P\in L(X,X)$ such that \eqref{eq:main_inequality} holds, we will additionally assume that $A$ satisfies a spectral-decomposition assumption.

\begin{definition}{\cite[Def.\ 5.2.5]{Curtain1995}}
\label{def:sda}
Denoting $\sigma^+(A) := \sigma(A)\cap \left\{ s \in
\mathbb{C}\!:\! \operatorname{Re}s\! \geq\! 0 \right\}$ and $\sigma^-(A)$
:= $\sigma(A)\cap \left\{ s \in \mathbb{C}\!:\! \operatorname{Re}s\! <\!
0 \right\}$, an operator $A$ satisfies the spectral decomposition assumption
if $\sigma^+(A)$ is bounded and separated from $\sigma^-(A)$ in such
a way that a rectifiable, simple, closed curve $\Gamma$ can be drawn
so as to enclose an open set containing $\sigma^+(A)$ in its
interior and $\sigma^-(A)$ in its exterior.
\end{definition}

\begin{remark} \label{rem:sda}
Classes of operators satisfying the spectrum decomposition assumption
include, e.g., delay equations \cite[Sec.\ 2.4]{Curtain1995} and
Riesz-spectral operators with a pure point spectrum and only finitely
many eigenvalues in $\sigma^+(A)$. More concrete examples of the latter
are compact perturbation of the Laplace operator, i.e., $A = \Delta +
c^2I$ for $c \in \mathbb{R}$ or models of damped vibrations such as 
$$
A = 
\begin{bmatrix}
0 & I \\ -A_0 & -D
\end{bmatrix}
$$
where $A_0$ is a positive operator and $D$ is an unbounded damping operator
(see, e.g., \cite{JacTru08} and the Euler-Bernoulli example with
Kelvin-Voigt damping).
\end{remark}

If $A$ satisfies the spectrum decomposition assumption, by
\cite[Lem.\ 2.5.7]{Curtain1995} the decomposition of the spectrum
induces a corresponding decomposition of $X$. Defining the spectral
projection $\mathcal{P}$ by 
$$
\mathcal{P}x := \frac{1}{2\pi i} \int\limits_{\Gamma} (sI - A)^{-1}x\,ds
$$
for $x\in X$, where $\Gamma$ from Definition \ref{def:sda} is traversed once in the
positive direction, we obtain the decomposition $X = X_u \oplus X_s$,
where $X_u = \mathcal{P}X$ and $X_s = (I-\mathcal{P})X$.  An important property is that $X_u \subset D(A)$, cf.\ \cite[Theorem 2.5.7 b)]{Curtain1995}. Moreover, the spectral projection yields a linear coordinate transform such that the pair $(A,C)$ can be transformed into the form
\begin{equation}
\label{eq:sd}
\widetilde{A} = 
\begin{bmatrix}
A_u & 0 \\ 0 & A_s
\end{bmatrix},  \qquad
\widetilde{C} = 
\begin{bmatrix}
C_u & C_s
\end{bmatrix}
\end{equation}
where $A_u, C_u, A_s, C_s$ are restrictions of $A$ and $C$ to $X_u$ and $X_s$, respectively. In particular, $A_u$ and $C_u$ are bounded operators. We impose the following assumption on $A$ and the corresponding spectral decomposition.

\begin{assumption} \label{as:sda}
Let the following hold:
\begin{itemize}
	\item $A$ satisfies the spectrum decomposition assumption such that it has
	the decomposition according to \eqref{eq:sd} and $A_s$ is exponentially stable.
	\item $\sigma^+(A)= \sigma(A)\cap \left\{ s \in
	\mathbb{C}\!:\! \operatorname{Re}s\! \geq\! 0 \right\}$ consists of finitely many eigenvalues of finite order.
	\item $X_u$ is orthogonal to $X_s$, i.e., $\langle x_u,x_s\rangle = 0$ for $x_u\in X_u$, $x_s\in X_s$.
\end{itemize}
\end{assumption}

If $C$ has finite rank, a spectrum decomposition assumption with stable part $A_s$ and finite dimensional observable part $(A_u,C_u)$ is equivalent to exponential detectability of $(A,C)$, cf.\ \cite[Theorem 5.2.7]{Curtain1995}.

\begin{theorem}
	\label{thm:2}
	Let Assumption~\ref{as:sda} hold, and $(A,C)$ be exponentially detectable. Then there is a positive and self-adjoint operator $P\in L(X,X)$ such that \eqref{eq:main_inequality} holds.
\end{theorem}
\begin{proof}
	We start with the verification of \eqref{eq:main_inequality} separately for the two subspaces, i.e., first for $x\in D(A)\cap X_s$ and second for $x\in  D(A)\cap X_u=X_u$. The result for general $x\in D(A)$ then follows by the orthogonality assumption.
	
	We first consider the stable subspace $X_s$. As the semigroup generated by $A_s$ on $X_s$ is stable, by \cite[Part II-1, Theorem 2.4]{Bensoussan2007}, there is a positive self-adjoint $P_s\in L(X_s,X_s)$ such that for all $x,y\in D(A)$
	\begin{align*}
	\langle \left(P_sA_s+A_s^*P_s\right)x,y\rangle = -\langle x,y\rangle.
	\end{align*}
	Hence, in particular,
	\begin{align*}
	\|x\|^2 &= \langle \left(-P_sA_s-A_s^*P_s\right)x,x\rangle \\&\leq \langle \left(-P_sA_s-A_s^*P_s\right)x,x\rangle + \|Cx\|_Y^2\\&= \langle \left(-P_sA_s-A_s^*P_s + C_s^*C_s\right)x,x\rangle,
	\end{align*}
	i.e, \eqref{eq:main_inequality} holds with $m=1$.
	
	Considering the unstable subspace $X_u$, we have by definition of the projection that $X_u\subset D(A)$, cf.\ \cite[Lemma 2.5.7 c]{Curtain1995} and hence the restriction $A_{|X_u}=A_u$ is bounded. Second, the subspace $X_u$ is finite dimensional, as it is spanned by finitely many (generalized) eigenvectors of $A$ of finite order. Thus, together with detectability, we can invoke \cite[Lemma 5.2]{GruG21} and obtain the existence of a positive definite matrix $P_u \in L(X_u,X_u)$ such that
	$$
	C_u^\top C_u-A_u^\top P_u-P_uA_u>0
	$$
	and hence, as $X_u$ is finite dimensional,
	$$
	\langle \left(C_u^\top C_u-A_u^\top P_u-P_uA_u\right) x,x\rangle >\alpha\|x\|^2
	$$
	for some $\alpha > 0$ and all $x\in X_u$. Setting 
	$$P=\begin{pmatrix}
	P_u&0\\
	0&P_s
	\end{pmatrix}$$ we obtain the desired inequality \eqref{eq:main_inequality} by orthogonality.
\end{proof}

\begin{remark}
If $\|Cx\|\geq c\|x\|$ for all $x\in \ker C$ and $A$ is bounded on $X$, one can prove the previous theorem by decomposing the space $X$ into $\ker C$ and $\ker C^\perp$.
\end{remark}
\section{Example}

In this part, we provide an application of Theorem~\ref{thm:2} to a heat equation with homogeneous Dirichlet boundary conditions adapted from \cite[Example 5.2.8]{Curtain1995}. To this end, set $X=L_2(0,1)$, $k\in(0,\pi^2)$ and 
\begin{align*}
A &:= \frac{d^2}{d\xi^2}+kI \\ D(A)&:= \{x \in L_2(0,1)\,|\, x,\tfrac{dx}{d\xi} \text{ are absolutely continuous},\\ &\qquad \tfrac{d^2x}{d\xi^2}\in L_2(0,1) \text{ and } \tfrac{dx}{d\xi}(0)=\tfrac{dx}{d\xi}(1)=0\}.
\end{align*}
The input and output space are set to be $Y=U=\mathbb{C}$ and the corresponding control and observation operators are given by
\begin{align*}
\quad  Bu &= bu \quad &&\text{ with }b(\xi)=\frac{1}{2\varepsilon}1_{[\xi_c-\varepsilon,\xi_c+\varepsilon]}(\xi),\\
Cx &= \int_0^1 c(\xi)x(\xi)\,d\xi \quad &&\text{ with }c(\xi)=\frac{1}{2\nu}1_{[\xi_o-\nu,\xi_o+\nu]}(\xi),
\end{align*}
where $\xi_c,\xi_o\in (0,1)$, $1_{S}(x)$ denotes the characteristic function of a set $S\subseteq(0,1)$ and $\varepsilon,\nu > 0$. One can easily see that $A$ is self-adjoint and $\sigma(A)=\{-n^2\pi^2 + k\,|\, n=0,1,\ldots\}$ with corresponding eigenvectors $\{1,\sqrt{2}\cos(n\pi\xi),n\geq 2\}$, all of multiplicity one. The eigenfunctions are pairwise orthogonal with respect to the standard scalar product in $L_2(0,1)$. Further, $\sigma_+(A)=\{k\}$ and hence $X_u=\operatorname{span}\{1(\cdot)\}$. This implies that Assumption~\ref{as:sda} is satisfied.

We proceed to show that $(A,C)$ is detectable. We chose the output injection operator $L=\left(\begin{smallmatrix}
(-k+\rho)1(\cdot)\\0
\end{smallmatrix}\right)$ for some $\rho>0$. As
\begin{align*}
A_ux = k \langle x, 1\rangle 1(\cdot)
\end{align*} we compute
\begin{align*}
&A_u-(k+\rho)1(\cdot)C_u\\
&=A_u -(k+\rho)1(\cdot) \int_0^1 c(\xi)\langle x,1\rangle 1(\xi)\,d\xi\\
&=A_u - (k+\rho) \langle x,1\rangle 1(\cdot) \underbrace{\int_0^1 c(\xi) 1(\xi)\,d\xi}_{=1}
&= -\rho \langle x, 1\rangle 1(\cdot),
\end{align*}
i.e., the eigenvalue of $A+LC$ with the largest real part is $-\rho$. This implies that $A+LC$ is stable, i.e., $(A,C)$ is detectable.

In particular, the assumptions of Theorem~\ref{thm:2} are satisfied, i.e., there is an operator $P\in L(X,X)$ such that 
\begin{align*}
\langle (Q - A^*P - PA)x,x\rangle \geq m\|x\|_X^2 \qquad \forall x\in D(A).
\end{align*}
Moreover, as $0\in \rho(A)$, we can also apply Theorem~\ref{thm:1} to deduce dissipativity in the sense of \eqref{eq:sd}.

It can easily be seen that this example can be extended to the case of $k\in \mathbb{R}$ arbitrary as long as $k\neq n^2\pi^2$ for some $n=0,1,\ldots$ such that $0\in \rho(A)$.

\section{Conclusion}
We proposed first steps towards characterizations of strict dissipativity in optimal control of infinite dimensional systems. We derived a characterization for strict dissipativity via an ellipticity condition and deduced a sufficient condition with detectability. The main tool for the second result was a spectral decomposition assumption that allowed to decompose the state space into infinite dimensional stable dynamics and finite dimensional unstable dynamics. Finally, we presented an example with linear unstable heat equation to illustrate our results. As stated in the introduction, concerning future work, it is desirable to generalize Theorem~\ref{thm:2}, i.e., in particular aim to remove the second and third assumption in Assumption~\ref{as:sda}. Further, a necessary condition via detectability which is available in the finite dimensional setting, is of interest. 

\bibliographystyle{abbrv}
\bibliography{references}   
\end{document}